\documentclass[preprint,12pt]{elsarticle}
\usepackage{amsmath}
\usepackage{float}
\usepackage{color,soul}
\usepackage{amsfonts}
\usepackage{amssymb}
\newtheorem{theorem}{Theorem}
\newtheorem{lemma}{Lemma}
\newtheorem{definition}{Definition}
\newtheorem{proposition}{Proposition}
\newtheorem{example}{Example}

\newtheorem{corollary}{Corollary}
\newtheorem{remark}{Remark}
\newtheorem{proof}{Proof}
\usepackage[left=2cm,right=2cm,top=2cm,bottom=2cm]{geometry}
\usepackage{hyperref}
\hypersetup{colorlinks,linkcolor={blue},citecolor={blue},urlcolor={red}} 
\usepackage[format=hang]{caption}
\usepackage{xcolor}

\newcommand{\R}{\mathbb {R}}
\newcommand{\N}{\mathbb {N}}

\newcommand{\C}{\mathbb {C}}

\usepackage{ae}
\usepackage{graphicx} 
\usepackage{booktabs} 
\usepackage[anythingbreaks]{breakurl}

\begin{document}

\begin{frontmatter}

\title{Chaotic Dynamics of Conformable Semigroups via Classical Theory}

\author[label1]{Mohamed Khoulane}
\author[label1]{Aziz El Ghazouani}
\author[label1]{M'hamed Elomari}
\affiliation[label1]{organization={Laboratory of Applied Mathematics and Scientific Computing, Sultan Moulay Slimane University},
            city={Beni Mellal},
            postcode={23000}, 
            country={Morocco}}

\begin{abstract} Conformable derivatives involve a fractional parameter while preserving locality: on smooth functions they reduce to a classical derivative multiplied by an explicit weight. Exploiting this structural feature, we show that conformable time evolution does not give rise to a genuinely new semigroup theory. Rather, it can be fully interpreted as a classical $C_0$--semigroup observed through a nonlinear change of time. For $\delta\in(0,1]$, we introduce the conformable clock
\[
\Psi(t)=\frac{t^\delta}{\delta},
\]
and prove that every $C_0$--$\delta$--semigroup $\mathcal S_\delta$ admits the representation
\[
\mathcal S_\delta(t)=\mathcal T(\Psi(t)),
\]
where $\mathcal T$ is a uniquely determined classical $C_0$--semigroup on the same state space. This correspondence is exact at the infinitesimal level: the $\delta$--generator of $\mathcal S_\delta$ coincides with the generator of $\mathcal T$ on a common domain, and conformable mild solutions are in one-to-one correspondence with classical mild solutions under the reparametrization $s=\Psi(t)$. In particular, orbit sets are unchanged by the conformable clock, so orbit-based linear dynamical properties are invariant; $\delta$--hypercyclicity and $\delta$--chaos coincide with their classical counterparts. As an application, we derive a conformable version of the Desch--Schappacher--Webb chaos criterion by transporting the classical result. The analysis is carried out in conformable Lebesgue spaces $L^{p,\delta}$, which are shown to be isometrically equivalent to standard $L^p$ spaces, allowing a direct transfer of estimates and spectral arguments. Altogether, the results clarify which dynamical features of conformable models are intrinsic and which arise solely from a nonlinear change of time.
\end{abstract}

\end{frontmatter}
\section{Introduction}

Fractional-type models are now widely used to describe anomalous transport, nonstandard relaxation, and memory effects in evolution problems. In the classical fractional framework, such phenomena are encoded through nonlocal operators of integral type, such as the Riemann--Liouville or Caputo derivatives. While these operators lead to genuinely new analytical features, they also introduce long-range temporal memory and a substantial increase in technical complexity.

In recent years, \emph{conformable calculus} has been proposed as an alternative fractional formalism. Its main attraction lies in its locality: for smooth functions, the conformable derivative reduces to a classical derivative multiplied by an explicit weight. This feature makes the approach appealing in operator theory and in the analysis of partial differential equations. At the same time, it raises a fundamental question: to what extent do conformable models generate dynamics that are genuinely different from those produced by classical evolution families?

Several works have pointed out that conformable derivatives lack the intrinsic memory effects characteristic of true fractional operators. This observation suggests that conformable dynamics should be understood through an explicit structural mechanism rather than by ad hoc computations. The purpose of the present paper is to provide a precise operator--theoretic answer to this issue. We show that conformable time evolution does not define a new semigroup theory, but can be described exactly as a classical $C_0$--semigroup evolving under a nonlinear change of time.

Throughout the paper we fix a single conformable order $\delta\in(0,1]$. All conformable objects (derivative, integral, semigroup, and function spaces) are indexed by this same parameter. The nonlinear time reparametrization associated with conformable calculus is given by the map
\[
\Psi(t)=\frac{t^{\delta}}{\delta},\qquad t\ge0,
\]
which is continuous, strictly increasing, and onto $[0,\infty)$. We refer to $\Psi$ as the \emph{conformable clock}. Our first main result shows that, for every conformable $C_0$--$\delta$--semigroup $\mathcal{S}_\delta=\{\mathcal{S}_\delta(t)\}_{t\ge 0}$ on a Banach space $X$, there exists a unique classical $C_0$--semigroup $\mathcal{T}= \{\mathcal{T}(s)\}_{s\ge 0}$ on the same space such that 

\begin{equation}\label{eq:intro_clock}
\mathcal S_\delta(t)=\mathcal T(\Psi(t)),\qquad t\ge0.
\end{equation}

Thus conformable evolution is nothing but classical evolution observed along a distorted time scale. The correspondence is exact at the infinitesimal level: the $\delta$--generator of $S_\delta$ coincides with the classical generator of $\mathcal T$ on a common domain, and mild solutions in conformable time are in one-to-one correspondence with classical mild solutions under the change of variables $s=\Psi(t)$.

Once this clock change is made explicit, a wide range of conformable results can be obtained directly from the classical theory. Instead of redeveloping arguments in the conformable variable, one may pass to the associated semigroup $\mathcal T$, use the established machinery of $C_0$--semigroups, and return to $\mathcal S_\delta$ through \eqref{eq:intro_clock}. This viewpoint yields transparent proofs for generation, well-posedness, dissipativity, and stability, and explains why many conformable results mirror their classical counterparts.

A second objective of the paper is to clarify the impact of the conformable clock on linear dynamics. Orbit-type properties such as hypercyclicity and Devaney chaos play a central role in the theory of $C_0$--semigroups. We show that these notions are invariant under the reparametrization $t\mapsto\Psi(t)$: $\delta$--hypercyclicity and $\delta$--chaos of $\mathcal S_\delta$ are equivalent to classical hypercyclicity and chaos of $\mathcal T$.
As a consequence, classical chaoticity criteria can be transported to conformable time without additional spectral analysis. In particular, we obtain a conformable version of the Desch--Schappacher--Webb criterion by pushing the classical theorem through the clock relation \eqref{eq:intro_clock}.

The conformable integral naturally induces the weighted measure $d\mu_\delta(t)=t^{\delta-1}\,dt$ and the spaces $L^{p,\delta}$ and $W^{m,p}_\delta$. We show that these spaces are explicitly, and in the relevant formulations isometrically, equivalent to the classical Lebesgue and Sobolev scales via the change of variables $s=\Psi(t)$. This identification provides a natural functional framework for generators, energy spaces, and boundary conditions, and allows one to transfer classical estimates and spectral arguments without loss.

Finally, we illustrate the transfer principle on conformable diffusion--transport equations posed on weighted Hilbert spaces. A nonlinear change of the spatial variable reduces the conformable operator to a classical drift--diffusion operator, yielding a unitary equivalence between the two realizations. Chaotic behavior of the conformable semigroup then follows immediately from known classical results, without performing new spectral computations. This example highlights the practical content of the conjugacy mechanism developed here: whenever a conformable model can be reduced to a classical one by an explicit change of variables, qualitative dynamical properties can be imported directly.

The contribution of this work is not a mere reformulation of known properties of conformable derivatives. It establishes a precise structural equivalence between conformable evolution and classical semigroup theory, showing that conformable dynamics does not produce a genuinely new fractional behavior. Unlike integral-type fractional models, where nonlocal memory effects are intrinsic, conformable calculus remains local and its ``fractional'' character is entirely encoded in a deterministic reparametrization of time. Our results make this mechanism explicit at every level: generators, mild solutions, functional spaces, and linear dynamics. In this sense, the paper settles the conceptual status of conformable evolution within semigroup theory and places it on a definitive analytical foundation.

The remainder of the paper is organized as follows. Section~2 recalls the elements of conformable calculus needed in the sequel and introduces conformable $C_0$--$\delta$--semigroups. Section~3 develops conformable Lebesgue and Sobolev spaces and establishes their explicit identification with the classical scales. Section~4 introduces the conformable energy space $H^1_{\delta,0}(\mathbb{R}_+)$ and proves an integration by parts formula. In Section~5 we study mild solutions and make the time-change principle precise, deriving generation results via the Lumer--Phillips theorem. Section~6 is devoted to linear dynamics: we prove the invariance of orbit-type notions under the conformable clock and deduce a conformable Desch--Schappacher--Webb criterion. Finally, Section~7 presents applications to conformable diffusion--transport equations and formulates a general conjugacy principle for transferring chaoticity from classical to conformable settings.

\section{Preliminary framework}\label{sec2}
 This section collects the notions that will be used throughout the paper. We restrict ourselves to those elements of conformable calculus that are essential for the operator--theoretic constructions developed later. A guiding principle is that conformable operators preserve \emph{locality}: for sufficiently smooth functions they reduce to classical derivatives or integrals multiplied by explicit weights. This feature sharply distinguishes conformable calculus from nonlocal fractional models of integral type \cite{kilbas2006,mainardi2010,tarasov2018} and is the key reason why a semigroup approach is available in the present setting.

The conformable framework originates from the definition introduced in \cite{khalil2014} and subsequently developed in \cite{abdeljawad2015,abdeljawad2016}. The notion of generalized semigroups associated with fractional time scales
was formalized in \cite{abdeljawad2015semigroup}. Our purpose is to connect these ideas with the classical theory of $C_0$--semigroups \cite{pazy1983,pruss1993,engel2000}, which also provides the natural background for the dynamical aspects considered later \cite{desch1997}. The emphasis is on making explicit the structural mechanisms behind conformable evolution, rather than on reproducing the general theory of fractional operators.

\subsection{Conformable integration and differentiation}

We begin with the conformable integral, which naturally induces a weighted
measure on $\R_+$.

\begin{definition}[Conformable integral {\cite{khalil2014,abdeljawad2015}}]
Let $\delta\in(0,1]$ and let $g:\R_+\to\C$ be measurable.
For $0\le a<t$, the conformable integral of order $\delta$ is defined by
\[
(\mathcal{I}_\delta g)(t)
:=
\int_a^t g(\xi)\,\xi^{\delta-1}\,d\xi,
\]
whenever the integral exists.
\end{definition}

This definition suggests working with the weighted measure
\[
d\mu_\delta(\xi):=\xi^{\delta-1}\,d\xi,
\]
and, unless otherwise specified, all integrals on $\R_+$ will be understood
with respect to $d\mu_\delta$.

The corresponding notion of differentiation is given below.

\begin{definition}[Conformable derivative {\cite{khalil2014,abdeljawad2015}}]
Let $\delta\in(0,1]$ and let $v:\R_+\to\C$.
The conformable derivative of order $\delta$ at $t>0$ is defined by
\[
\mathcal{D}_t^\delta v(t)
:=
\lim_{h\to 0}
\frac{v\bigl(t+h\,t^{1-\delta}\bigr)-v(t)}{h},
\]
provided the limit exists.
For $\delta=1$, this reduces to the usual derivative.
\end{definition}

The local nature of this operator is made explicit by the following
representation property; see \cite{abdeljawad2015,abdeljawad2016}.

\begin{proposition}
Let $v\in C^1((0,\infty))$ and $\delta\in(0,1]$. Then, for every $t>0$,
\[
\mathcal{D}_t^\delta v(t)
=
t^{1-\delta}\,\frac{dv}{dt}(t).
\]
\end{proposition}

Thus conformable differentiation acts as a weighted classical derivative and
does not involve memory effects.
This structural feature will play a decisive role in the analysis of the
associated evolution families.

\subsection{Conformable semigroups}

Conformable calculus also leads to a generalized notion of one--parameter
semigroups obtained through a nonlinear rescaling of time
\cite{abdeljawad2015semigroup}.

\begin{definition}[Conformable $\delta$--semigroup]
Let $\delta\in(0,a]$ for some $a>0$, and let $X$ be a Banach space.
A family of bounded linear operators
\(
\mathcal{S}_\delta=\{\mathcal{S}_\delta(t)\}_{t\ge0}\subset\mathcal{L}(X)
\)
is called a \emph{conformable $\delta$--semigroup} if:
\begin{enumerate}
\item $\mathcal{S}_\delta(0)=\mathrm{Id}$ on $X$;
\item for all $r,q\ge0$,
\[
\mathcal{S}_\delta\bigl((r+q)^{1/\delta}\bigr)
=
\mathcal{S}_\delta\bigl(r^{1/\delta}\bigr)\,
\mathcal{S}_\delta\bigl(q^{1/\delta}\bigr).
\]
\end{enumerate}
\end{definition}

For $\delta=1$ this reduces to the classical semigroup law
\cite{pazy1983,engel2000}.

\begin{example}
Let $B\in\mathcal{L}(X)$ be bounded and define
\[
\mathcal{S}_{1/2}(t)=\exp\!\bigl(2\,t^{1/2}B\bigr),\qquad t\ge0.
\]
Then $\mathcal{S}_{1/2}$ is a conformable $\tfrac12$--semigroup.
\end{example}

\begin{example}
Let $X=C([0,\infty))$ endowed with the supremum norm and set
\[
(\mathcal{S}_{1/2}(t)f)(x)=f\bigl(x+2\,t^{1/2}\bigr),
\qquad x,t\ge0.
\]
Then $\mathcal{S}_{1/2}$ defines a $\tfrac12$--semigroup of bounded linear
operators on $X$.
\end{example}

\begin{definition}[$C_0$--$\delta$--semigroup]
A conformable $\delta$--semigroup $\mathcal{S}_\delta$ is said to be a
\emph{$C_0$--$\delta$--semigroup} if
\[
\lim_{t\to0^+}\mathcal{S}_\delta(t)x=x
\quad\text{for all }x\in X.
\]
\end{definition}

The infinitesimal behavior of $\mathcal{S}_\delta$ at the origin is described by
its $\delta$--generator, defined by
\[
D(\mathcal{A}_\delta)
=
\Bigl\{
x\in X:\;
\lim_{t\to0^+}\mathcal{D}_t^\delta\mathcal{S}_\delta(t)x
\ \text{exists}
\Bigr\},
\qquad
\mathcal{A}_\delta x
=
\lim_{t\to0^+}\mathcal{D}_t^\delta\mathcal{S}_\delta(t)x.
\]

\begin{remark}
In the limiting case $\delta=1$, one recovers the classical framework of
$C_0$--semigroups and their generators \cite{pazy1983,engel2000}.
\end{remark}

\section{Conformable Lebesgue and Sobolev spaces}\label{sec3}

The conformable derivative modifies the time scale in a non-uniform way.
At the level of integration, this deformation appears through the weight
$t^{\delta-1}$ (with $\delta\in(0,1]$), which is already present in the
conformable integral introduced in Section~\ref{sec2}:
\[
(\mathcal I_\delta f)(t)=\int_0^t f(s)\,s^{\delta-1}\,ds.
\]
It is therefore natural to work on $\R_+=(0,\infty)$ with the weighted measure
\[
d\mu_\delta(t):=t^{\delta-1}\,dt.
\]
This choice incorporates the conformable change of time into the functional
setting and is well suited for energy methods, compactness arguments, and the
analysis of generators arising in conformable evolution equations.
In the present section we introduce the corresponding Lebesgue and Sobolev
spaces and record the properties needed later.
Whenever a statement is a direct consequence of standard weighted $L^p$ theory,
we simply indicate it and refer to the classical literature.

\subsection{Conformable Lebesgue spaces}

Let $I\subset\R_+$ be an interval and fix $\delta\in(0,1]$.

\begin{definition}[Conformable Lebesgue spaces]\label{def:Lpdelta}
For $1\le p<\infty$, the conformable Lebesgue space on $I$ is defined by
\[
L^{p,\delta}(I)
:=
\Bigl\{
f:I\to\C \ \text{measurable} \,;\,
\int_I |f(t)|^p\,d\mu_\delta(t)<\infty
\Bigr\},
\]
endowed with the norm
\[
\|f\|_{p,\delta}
:=
\Bigl(\int_I |f(t)|^p\,d\mu_\delta(t)\Bigr)^{1/p}
=
\Bigl(\int_I |f(t)|^p\,t^{\delta-1}\,dt\Bigr)^{1/p}.
\]
\end{definition}

\begin{remark}[Weighted interpretation]\label{rem:weightedLp}
Since $t^{\delta-1}$ is locally integrable on $(0,\infty)$ for every $\delta>0$,
the measure space $(I,\mu_\delta)$ is $\sigma$--finite.
Hence $L^{p,\delta}(I)$ coincides with the classical weighted space
$L^p(I,t^{\delta-1}dt)$.
Standard results such as H\"older and Minkowski inequalities, density of simple
functions, and completeness follow from the general theory of weighted $L^p$
spaces; see, for instance, \cite{KufnerWeighted,AdamsFournier}.
\end{remark}

A key point for our purposes is that the weight $t^{\delta-1}$ can be removed by
a simple change of variables, which makes the conformable scale explicitly
comparable to the classical one.

\begin{proposition}[Isometric reduction to the classical $L^p$ scale]\label{prop:isometryLp}
Let $I=(0,T)$ with $0<T\le\infty$ and $\delta\in(0,1]$.
Define the increasing map
\[
\Psi(t):=\frac{t^\delta}{\delta}, \qquad t\in(0,T).
\]
Then
\[
U:L^{p,\delta}(0,T)\longrightarrow L^p\bigl(0,\Psi(T)\bigr),
\qquad
(Uf)(s)=f\bigl((\delta s)^{1/\delta}\bigr),
\]
is a linear isometric isomorphism.
\end{proposition}

\begin{proof}
The map $\Psi$ is a strictly increasing $C^1$--diffeomorphism from $(0,T)$ onto
$(0,\Psi(T))$ and satisfies $\Psi'(t)=t^{\delta-1}$.
With the change of variables $s=\Psi(t)$, we obtain
\[
\int_0^T |f(t)|^p\,t^{\delta-1}\,dt
=
\int_0^{\Psi(T)} \bigl|f\bigl((\delta s)^{1/\delta}\bigr)\bigr|^p\,ds,
\]
hence $\|Uf\|_{L^p(0,\Psi(T))}=\|f\|_{p,\delta}$.
Linearity is immediate, and surjectivity follows by setting
$f(t):=g(\Psi(t))$ for $g\in L^p(0,\Psi(T))$.
\end{proof}

\begin{proposition}[Completeness and Hilbert case]\label{prop:BanachHilbertLp}
For every $1\le p<\infty$, the normed space $\bigl(L^{p,\delta}(I),\|\cdot\|_{p,\delta}\bigr)$
is Banach.
Moreover, for $p=2$, the inner product
\[
(f,g)_{2,\delta}:=\int_I f(t)\overline{g(t)}\,d\mu_\delta(t)
\]
turns $L^{2,\delta}(I)$ into a Hilbert space.
\end{proposition}

\begin{proof}
This is a direct consequence of the completeness of weighted $L^p$ spaces on
$\sigma$--finite measure spaces; see \cite{KufnerWeighted}.
For $p=2$, the form $(\cdot,\cdot)_{2,\delta}$ is a well-defined inner product
whose induced norm is $\|\cdot\|_{2,\delta}$, and completeness yields the Hilbert
structure.
\end{proof}

\begin{remark}\label{rem:measure_encodes_clock}
Proposition~\ref{prop:isometryLp} shows that the conformable Lebesgue scale does
not introduce a new integrability regime: it encodes the conformable time
reparametrization at the level of the measure.
This viewpoint will be mirrored by the Sobolev scale defined below.
\end{remark}

\subsection{Conformable Sobolev spaces on $\R_+$}

We now introduce Sobolev spaces associated with conformable derivatives.
Throughout this subsection, $\mathcal D_t^\delta$ denotes the conformable
derivative defined in Section~\ref{sec2}.
For smooth functions $u\in C^1((0,\infty))$, recall that
\[
\mathcal D_t^\delta u(t)=t^{1-\delta}u'(t),
\]
which emphasizes the weighted nature of the corresponding norms.

\begin{definition}[Conformable Sobolev spaces]\label{def:Wmpdelta}
Let $m\in\N$ and $1\le p<\infty$.
We define
\[
W^{m,p}_\delta(\R_+)
:=
\Bigl\{
u\in L^{p,\delta}(\R_+)\,;\,
\mathcal D_t^{k\delta}u\in L^{p,\delta}(\R_+)
\ \text{for all }k=1,\dots,m
\Bigr\},
\]
endowed with the norm
\[
\|u\|_{W^{m,p}_\delta}
:=
\Bigl(\sum_{k=0}^m \|\mathcal D_t^{k\delta}u\|_{p,\delta}^p\Bigr)^{1/p},
\qquad
\mathcal D_t^{0\cdot\delta}u:=u.
\]
\end{definition}

\begin{remark}[On the choice of definition]\label{rem:sobolev_choice}
The derivatives $\mathcal D_t^{k\delta}u$ are interpreted as iterated conformable
derivatives, taken pointwise for smooth functions and in the corresponding weak
sense in $L^{p,\delta}$.
In this work, $W^{m,p}_\delta(\R_+)$ is used as a functional framework tailored to
the local conformable calculus.
A systematic comparison with alternative Sobolev-type constructions is outside
our scope; we refer the reader to the literature on weighted Sobolev spaces for
general background (e.g. \cite{KufnerWeighted,AdamsFournier}).
\end{remark}

\begin{proposition}[Completeness]\label{prop:BanachWmp}
For all $m\in\N$ and $1\le p<\infty$, the space $W^{m,p}_\delta(\R_+)$ is a Banach space.
\end{proposition}

\begin{proof}
Let $(u_n)$ be Cauchy in $W^{m,p}_\delta(\R_+)$.
Then for each $k=0,\dots,m$, the sequence $(\mathcal D_t^{k\delta}u_n)$ is Cauchy in
$L^{p,\delta}(\R_+)$, hence converges to some $v_k\in L^{p,\delta}(\R_+)$.
Setting $u:=v_0$, we obtain $u_n\to u$ in $L^{p,\delta}(\R_+)$.
By the definition of weak conformable derivatives, the limits satisfy
$v_k=\mathcal D_t^{k\delta}u$, so $u\in W^{m,p}_\delta(\R_+)$ and $u_n\to u$ in the
Sobolev norm.
\end{proof}

\subsubsection*{Hilbert case}

\begin{definition}[Conformable Hilbert--Sobolev spaces]\label{def:Hmdelta}
For $m\in\N$, we set
\[
H_\delta^m(\R_+):=W^{m,2}_\delta(\R_+),
\qquad
\|u\|_{H_\delta^m}^2
=
\sum_{k=0}^m \|\mathcal D_t^{k\delta}u\|_{2,\delta}^2.
\]
\end{definition}

\begin{proposition}
For every $m\in\N$, the space $H_\delta^m(\R_+)$ is a Hilbert space with inner product
\[
(u,v)_{H_\delta^m}
:=
\sum_{k=0}^m
\bigl(\mathcal D_t^{k\delta}u,\mathcal D_t^{k\delta}v\bigr)_{2,\delta}.
\]
\end{proposition}

\begin{proof}
Each term defines an inner product on $L^{2,\delta}(\R_+)$, and the finite sum is
again an inner product on $H_\delta^m(\R_+)$.
Completeness follows from Proposition~\ref{prop:BanachWmp} with $p=2$.
\end{proof}

\subsubsection*{Vanishing trace subspace}

\begin{definition}[The space $H_{\delta,0}^1(\R_+)$]\label{def:H10delta}
We define
\[
H_{\delta,0}^1(\R_+)
:=
\overline{C_c^\infty(0,\infty)}^{\,H_\delta^1}.
\]
\end{definition}

\begin{remark}\label{rem:H10_role}
The space $H_{\delta,0}^1(\R_+)$ can be viewed as the conformable analogue of the
classical $H_0^1$ on a half-line.
It is adapted to homogeneous boundary conditions at $t=0$ and will play a central
role in the construction of generators and energy identities in the next section.
\end{remark}

\begin{remark}[Local versus nonlocal fractional settings]\label{rem:local_vs_nonlocal}
The above constructions rely crucially on the locality of $\mathcal D_t^\delta$ and
on the weight encoded in $\mu_\delta$.
They should not be confused with Sobolev frameworks tailored to nonlocal fractional
operators of integral type, where memory effects are structural
\cite{kilbas2006,mainardi2010,tarasov2018}.
\end{remark}

\section{Mild solutions, time reparametrization, and generation of conformable semigroups}
\label{sec:mild-lumer}

This section develops a semigroup-based approach to conformable evolution
equations.
Two complementary perspectives will be used repeatedly:
\begin{itemize}
\item[(i)] the \emph{mild formulation} generated by a $C_0$--$\delta$--semigroup,
which provides a natural notion of solution in Banach spaces;
\item[(ii)] a \emph{time-change interpretation} showing that conformable dynamics
is classical $C_0$--semigroup evolution observed along a nonlinear clock.
\end{itemize}
The second viewpoint is particularly useful, since it allows one to transport
generation and stability results from the standard theory to the conformable
setting. In particular, it gives direct access to the Lumer--Phillips criterion
in Hilbert spaces.

Throughout, $X$ denotes a Banach space (and later a Hilbert space when required),
and
\[
\mathcal S_\delta=\{\mathcal S_\delta(t)\}_{t\ge0}\subset\mathcal L(X),
\qquad \delta\in(0,1],
\]
is a $C_0$--$\delta$--semigroup with $\delta$--generator $\mathcal A_\delta$.

\subsection*{Time reparametrization and the induced classical semigroup}

A basic feature of conformable calculus is the nonlinear change of time
\[
s=\Psi(t):=\frac{t^\delta}{\delta},\qquad t\ge0,
\]
which is continuous, strictly increasing, and maps $[0,\infty)$ onto itself.
We will refer to $\Psi$ as the \emph{conformable clock}.
Using this reparametrization, we associate to $\mathcal S_\delta$ the operator family
\[
\mathcal T(s):=\mathcal S_\delta\!\bigl(\Psi^{-1}(s)\bigr)
=\mathcal S_\delta\!\bigl((\delta s)^{1/\delta}\bigr),
\qquad s\ge0.
\]

\begin{lemma}\label{lem:TisC0_combined}
The family $\{\mathcal T(s)\}_{s\ge0}$ is a (classical) $C_0$--semigroup on $X$.
\end{lemma}

\begin{proof}
\emph{Identity at $0$.}
Since $\Psi^{-1}(0)=0$ and $\mathcal S_\delta(0)=I$, we have $\mathcal T(0)=I$.

\smallskip
\emph{Semigroup law.}
Fix $s,r\ge0$ and set $t_1=\Psi^{-1}(s)=(\delta s)^{1/\delta}$ and
$t_2=\Psi^{-1}(r)=(\delta r)^{1/\delta}$.
Then
\[
\Psi^{-1}(s+r)=(\delta(s+r))^{1/\delta}=(t_1^\delta+t_2^\delta)^{1/\delta}.
\]
Using the defining composition rule of a $\delta$--semigroup,
\[
\mathcal T(s+r)
=\mathcal S_\delta\!\bigl((t_1^\delta+t_2^\delta)^{1/\delta}\bigr)
=\mathcal S_\delta(t_1)\,\mathcal S_\delta(t_2)
=\mathcal T(s)\,\mathcal T(r).
\]

\smallskip
\emph{Strong continuity at $0$.}
Let $x\in X$. As $s\to0^+$, $\Psi^{-1}(s)=(\delta s)^{1/\delta}\to0$.
Since $\mathcal S_\delta$ is strongly continuous at $0$,
\[
\|\mathcal T(s)x-x\|
=\|\mathcal S_\delta(\Psi^{-1}(s))x-x\|\longrightarrow 0.
\]
Thus $\mathcal T$ is a $C_0$--semigroup.
\end{proof}

\subsection*{Mild solutions in conformable time and in classical time}

We now record the mild formulation in both time variables.

\begin{definition}[Mild solutions]\label{def:mild_combined}
\leavevmode
\begin{enumerate}
\item
A mapping $x:[0,\infty)\to X$ is called a \emph{mild solution} of the conformable
Cauchy problem
\[
\mathcal D_t^\delta x(t)=\mathcal A_\delta x(t),
\qquad x(0)=x_0,
\]
if it is given by
\[
x(t)=\mathcal S_\delta(t)x_0,\qquad t\ge0.
\]

\item
A mapping $y:[0,\infty)\to X$ is called a \emph{mild solution} of the classical
Cauchy problem
\[
y'(s)=\mathcal B y(s),\qquad y(0)=y_0,
\]
if it is given by
\[
y(s)=\mathcal T(s)y_0,\qquad s\ge0,
\]
where $\mathcal B$ denotes the generator of the $C_0$--semigroup $\mathcal T$.
\end{enumerate}
\end{definition}

\begin{remark}\label{rem:mild_general}
In both settings, the mild solution is defined through the evolution family.
This is the natural notion in Banach spaces when the initial datum is not
assumed to belong to the operator domain.
\end{remark}

\subsection*{Identification of generators through the clock change}

The next result shows that the conformable generator coincides with the
classical generator of the time-changed semigroup.

\begin{theorem}[Coincidence of generators]\label{thm:generators_combined}
Let $\mathcal A_\delta$ be the $\delta$--generator of $\mathcal S_\delta$, and let
$\mathcal B$ be the (classical) generator of $\mathcal T$.
Then
\[
D(\mathcal A_\delta)=D(\mathcal B),
\qquad
\mathcal A_\delta x=\mathcal Bx\quad\text{for all }x\in D(\mathcal A_\delta).
\]
\end{theorem}

\begin{proof}
Fix $x\in X$ and relate $s$ and $t$ by $t=\Psi^{-1}(s)=(\delta s)^{1/\delta}$.
By definition of $\mathcal T$,
\[
\frac{\mathcal T(s)x-x}{s}
=
\frac{\mathcal S_\delta(t)x-x}{\Psi(t)}
=
\frac{\mathcal S_\delta(t)x-x}{t^\delta/\delta}.
\]
Letting $s\to0^+$ is equivalent to letting $t\to0^+$.
Hence the limit defining $\mathcal Bx$ exists if and only if the limit defining
$\mathcal A_\delta x$ exists, and the two limits coincide.
Therefore $D(\mathcal A_\delta)=D(\mathcal B)$ and $\mathcal A_\delta x=\mathcal Bx$
on this common domain.
\end{proof}

\subsection*{Equivalence of mild solutions under reparametrization}

The time change provides a direct correspondence between the two mild
formulations.

\begin{proposition}\label{prop:correspondence_combined}
Let $x_0\in X$ and set $x(t)=\mathcal S_\delta(t)x_0$.
Define
\[
y(s):=x\bigl(\Psi^{-1}(s)\bigr)=x\bigl((\delta s)^{1/\delta}\bigr),
\qquad s\ge0.
\]
Then $y(s)=\mathcal T(s)x_0$, hence $y$ is the classical mild solution generated by
$\mathcal B$.
Conversely, if $y(s)=\mathcal T(s)x_0$, then
\[
x(t):=y(\Psi(t))=y(t^\delta/\delta),\qquad t\ge0,
\]
satisfies $x(t)=\mathcal S_\delta(t)x_0$ and is the conformable mild solution.
\end{proposition}

\begin{proof}
Using the definition of $\mathcal T$,
\[
y(s)=x((\delta s)^{1/\delta})
=\mathcal S_\delta((\delta s)^{1/\delta})x_0
=\mathcal T(s)x_0.
\]
For the converse, start from $y(s)=\mathcal T(s)x_0$ and set $s=\Psi(t)$:
\[
x(t)=y(\Psi(t))
=\mathcal T(\Psi(t))x_0
=\mathcal S_\delta\!\bigl(\Psi^{-1}(\Psi(t))\bigr)x_0
=\mathcal S_\delta(t)x_0.
\]
\end{proof}

\subsection*{Dissipativity and resolvent bounds in Hilbert spaces}

From this point on, we assume that $X=H$ is a complex Hilbert space with inner product
$\langle\cdot,\cdot\rangle$ and norm $\|\cdot\|$.

\begin{definition}[Dissipative and maximal dissipative operators]\label{def:mdiss}
A densely defined operator $B:D(B)\subset H\to H$ is called \emph{dissipative} if
\[
\Re\langle Bx,x\rangle\le 0
\quad\text{for all }x\in D(B).
\]
It is \emph{maximal dissipative} if there exists $\lambda>0$ such that
\[
\mathrm{Ran}(\lambda I-B)=H.
\]
\end{definition}

\begin{lemma}\label{lem:resolvent_combined}
Assume that $B$ is dissipative.
Then for every $\lambda>0$ and $x\in D(B)$,
\[
\|(\lambda I-B)x\|\ge \lambda\|x\|.
\]
Consequently, if $\lambda I-B$ is surjective, then $(\lambda I-B)^{-1}$ exists and
\[
\|(\lambda I-B)^{-1}\|\le \frac1\lambda.
\]
\end{lemma}

\begin{proof}
Fix $\lambda>0$ and $x\in D(B)$.
Dissipativity gives
\[
\Re\langle (\lambda I-B)x,x\rangle
=
\lambda\|x\|^2-\Re\langle Bx,x\rangle
\ge \lambda\|x\|^2.
\]
By Cauchy--Schwarz,
\[
\Re\langle (\lambda I-B)x,x\rangle
\le |\langle (\lambda I-B)x,x\rangle|
\le \|(\lambda I-B)x\|\,\|x\|.
\]
Combining the two inequalities yields
$\|(\lambda I-B)x\|\ge \lambda\|x\|$.
If $\lambda I-B$ is surjective, the inequality implies injectivity, hence
bijectivity, and the bound on the inverse follows.
\end{proof}

\subsection*{Lumer--Phillips and generation in conformable time}

We recall the Lumer--Phillips characterization in a form adapted to our needs.

\begin{theorem}[Lumer--Phillips]\label{thm:lumer_phillips_combined}
Let $B$ be a densely defined operator on $H$.
Then $B$ generates a contraction $C_0$--semigroup on $H$ if and only if $B$ is
maximal dissipative.
\end{theorem}

\begin{proof}
\emph{Necessity.}
Assume that $B$ generates a contraction semigroup $\{T(s)\}_{s\ge0}$.
For $x\in D(B)$, the orbit $s\mapsto T(s)x$ is differentiable and satisfies
$\frac{d}{ds}T(s)x=BT(s)x$ with $T(0)x=x$.
Let $\phi(s):=\|T(s)x\|^2$.
Differentiating at $s=0$ yields $\phi'(0)=2\Re\langle Bx,x\rangle$.
Since $\|T(s)\|\le1$, the function $\phi$ is nonincreasing near $0$, hence
$\Re\langle Bx,x\rangle\le0$.
Moreover, standard semigroup theory gives $\mathrm{Ran}(\lambda I-B)=H$ for every
$\lambda>0$, so $B$ is maximal dissipative.

\medskip
\emph{Sufficiency.}
Assume that $B$ is maximal dissipative.
Then $\mathrm{Ran}(\lambda I-B)=H$ for some $\lambda>0$, and
Lemma~\ref{lem:resolvent_combined} yields the resolvent estimate
$\|(\lambda I-B)^{-1}\|\le 1/\lambda$.
By the resolvent identity, the same estimate holds for all sufficiently large
$\lambda$, and the Hille--Yosida theorem implies that $B$ generates a contraction
$C_0$--semigroup.
\end{proof}

\begin{corollary}[Generation of conformable contraction semigroups]\label{cor:conformable_LP}
Let $\mathcal S_\delta$ be a $C_0$--$\delta$--semigroup on a Hilbert space $H$
with $\delta$--generator $\mathcal A_\delta$.
Assume that $\mathcal A_\delta$ is densely defined, dissipative, and that
\[
\mathrm{Ran}(\lambda I-\mathcal A_\delta)=H
\quad\text{for some }\lambda>0.
\]
Then $\mathcal S_\delta$ is contractive:
\[
\|\mathcal S_\delta(t)\|\le 1\qquad\text{for all }t\ge0.
\]
\end{corollary}

\begin{proof}
Let $\mathcal T(s)=\mathcal S_\delta((\delta s)^{1/\delta})$.
By Lemma~\ref{lem:TisC0_combined}, $\mathcal T$ is a classical $C_0$--semigroup.
Let $\mathcal B$ be its generator.
By Theorem~\ref{thm:generators_combined},
\[
\mathcal B=\mathcal A_\delta
\quad\text{and}\quad
D(\mathcal B)=D(\mathcal A_\delta).
\]
Hence $\mathcal B$ is densely defined, dissipative, and satisfies
$\mathrm{Ran}(\lambda I-\mathcal B)=H$.
By Theorem~\ref{thm:lumer_phillips_combined}, $\mathcal B$ generates a contraction
semigroup, so $\|\mathcal T(s)\|\le1$ for all $s\ge0$.
Finally,
\[
\mathcal S_\delta(t)=\mathcal T(\Psi(t))=\mathcal T(t^\delta/\delta),
\]
which yields $\|\mathcal S_\delta(t)\|\le1$ for all $t\ge0$.
\end{proof}

\begin{remark}
Corollary~\ref{cor:conformable_LP} provides a convenient stability criterion for
conformable Cauchy problems in Hilbert spaces.
In applications, dissipativity is typically obtained by combining a conformable
integration by parts identity with the boundary condition encoded in
$H_{\delta,0}^1(\R_+)$.
\end{remark}

\section{Linear dynamics under conformable time reparametrization}\label{sec5}

Let $X$ be a separable Banach space and fix $\delta\in(0,1]$.
Let
\(
\mathcal S_\delta=\{\mathcal S_\delta(t)\}_{t\ge0}\subset\mathcal L(X)
\)
be a $C_0$--$\delta$--semigroup in the sense of Section~\ref{sec2}.
A central point in conformable dynamics is that the parameter $t$ does not create
new trajectories: it only alters the \emph{pace} at which a given orbit is
traversed.
This simple observation allows one to transfer orbit-based properties
(hypercyclicity, density of periodic points, and chaos) between conformable and
classical semigroups, and to exploit established criteria from the classical
theory \cite{desch1997,pazy1983,engel2000}.
In particular, the conformable setting inherits the Desch--Schappacher--Webb
criterion once the time-change mechanism is made explicit.

\subsection{The conformable clock and the associated classical semigroup}

Define the \emph{conformable clock} by
\[
\Psi:[0,\infty)\to[0,\infty),\qquad \Psi(t)=\frac{t^\delta}{\delta}.
\]
Then $\Psi$ is continuous, strictly increasing, onto, and its inverse is
\[
\Psi^{-1}(s)=(\delta s)^{1/\delta},\qquad s\ge0.
\]
To $\mathcal S_\delta$ we associate the operator family
\begin{equation}\label{eq:defT}
\mathcal T(s)
:=
\mathcal S_\delta\!\bigl(\Psi^{-1}(s)\bigr)
=
\mathcal S_\delta\!\bigl((\delta s)^{1/\delta}\bigr),
\qquad s\ge0,
\end{equation}
so that, equivalently,
\begin{equation}\label{eq:conjugacy_time}
\mathcal S_\delta(t)=\mathcal T(\Psi(t)),\qquad t\ge0.
\end{equation}

\begin{lemma}\label{lem:T_C0}
The family $\{\mathcal T(s)\}_{s\ge0}$ is a (classical) $C_0$--semigroup on $X$.
\end{lemma}

\begin{proof}
\emph{(i) Identity at $0$.}
Since $\Psi^{-1}(0)=0$ and $\mathcal S_\delta(0)=I$, we get
$\mathcal T(0)=\mathcal S_\delta(0)=I$.

\smallskip
\emph{(ii) Semigroup law.}
Let $s,r\ge0$ and set $t_1=\Psi^{-1}(s)=(\delta s)^{1/\delta}$ and
$t_2=\Psi^{-1}(r)=(\delta r)^{1/\delta}$.
Then $t_1^\delta=\delta s$ and $t_2^\delta=\delta r$, hence
\[
\Psi^{-1}(s+r)=(\delta(s+r))^{1/\delta}=(t_1^\delta+t_2^\delta)^{1/\delta}.
\]
Using the defining composition rule for a $C_0$--$\delta$--semigroup,
\[
\mathcal S_\delta\!\bigl((t_1^\delta+t_2^\delta)^{1/\delta}\bigr)
=\mathcal S_\delta(t_1)\mathcal S_\delta(t_2),
\]
we obtain
\[
\mathcal T(s+r)
=\mathcal S_\delta(\Psi^{-1}(s+r))
=\mathcal S_\delta(t_1)\mathcal S_\delta(t_2)
=\mathcal T(s)\mathcal T(r).
\]

\smallskip
\emph{(iii) Strong continuity at $0$.}
Fix $x\in X$.
As $s\to0^+$, $\Psi^{-1}(s)=(\delta s)^{1/\delta}\to0^+$.
By strong continuity of $\mathcal S_\delta$ at $0$,
\[
\|\mathcal T(s)x-x\|
=
\|\mathcal S_\delta(\Psi^{-1}(s))x-x\|
\longrightarrow 0.
\]
Thus $\mathcal T$ is a $C_0$--semigroup.
\end{proof}

\begin{remark}[Orbit identity]\label{rem:orbit_identity}
From \eqref{eq:conjugacy_time} and the fact that $\Psi$ is a bijection of $[0,\infty)$,
we infer that for each $x\in X$,
\[
\{\mathcal S_\delta(t)x:\ t\ge0\}
=
\{\mathcal T(s)x:\ s\ge0\}.
\]
Hence conformable reparametrization changes only the parametrization of the orbit,
not the set of points visited along it.
\end{remark}

\subsection{Orbit-type notions for conformable semigroups}

We now introduce the dynamical notions used in the sequel.

\begin{definition}[$\delta$--hypercyclicity and $\delta$--chaos]\label{def:delta_hc_chaos}
\leavevmode
\begin{enumerate}
\item
The semigroup $\mathcal S_\delta$ is called \emph{$\delta$--hypercyclic} if there exists
$x\in X$ such that the orbit $\{\mathcal S_\delta(t)x:\ t\ge0\}$ is dense in $X$.
Such an $x$ is called a $\delta$--hypercyclic vector.

\item
It is called \emph{$\delta$--chaotic} if it is $\delta$--hypercyclic and its set of
$\delta$--periodic points
\[
X_p^\delta
:=
\{x\in X:\ \exists\,t>0\ \text{such that}\ \mathcal S_\delta(t)x=x\}
\]
is dense in $X$.
\end{enumerate}
\end{definition}

In the spirit of the classical semigroup dynamics framework \cite{desch1997}, we also set
\[
X_0^\delta
:=
\Bigl\{x\in X:\ \lim_{t\to\infty}\mathcal S_\delta(t)x=0\Bigr\},
\]
and
\[
X_\infty^\delta
:=
\Bigl\{x\in X:\ \forall\varepsilon>0\ \exists\,y\in X,\ t>0\ \text{such that}\
\|y\|<\varepsilon\ \text{and}\ \|\mathcal S_\delta(t)y-x\|<\varepsilon\Bigr\}.
\]
The set $X_\infty^\delta$ encodes the ability to approximate arbitrary target states
from arbitrarily small initial data.

\subsection{Dynamical sets are invariant under the conformable clock}

Let $X_0$, $X_\infty$ and $X_p$ denote the corresponding sets for the classical semigroup
$\mathcal T$:
\begin{align*}
X_0&=\Bigl\{x\in X:\lim_{s\to\infty}\mathcal T(s)x=0\Bigr\},\\
X_\infty&=\Bigl\{x\in X:\forall\varepsilon>0\ \exists\,y\in X,\ s>0:
\|y\|<\varepsilon,\ \|\mathcal T(s)y-x\|<\varepsilon\Bigr\},\\
X_p&=\Bigl\{x\in X:\exists\,s>0\ \text{with}\ \mathcal T(s)x=x\Bigr\}.
\end{align*}

\begin{lemma}[Invariance of dynamical sets]\label{lem:invariance_sets}
With $\mathcal T$ defined by \eqref{eq:defT}, one has
\[
X_0^\delta=X_0,\qquad X_\infty^\delta=X_\infty,\qquad X_p^\delta=X_p.
\]
In particular, $\mathcal S_\delta$ is $\delta$--hypercyclic (resp.\ $\delta$--chaotic)
if and only if $\mathcal T$ is hypercyclic (resp.\ chaotic) in the classical sense.
\end{lemma}

\begin{proof}
We use \eqref{eq:conjugacy_time} and the fact that $\Psi$ is a homeomorphism of
$[0,\infty)$ onto itself.

\smallskip
\noindent\emph{Step 1: invariance of $X_0^\delta$.}
For $x\in X$,
\[
\lim_{t\to\infty}\mathcal S_\delta(t)x=0
\ \Longleftrightarrow\
\lim_{t\to\infty}\mathcal T(\Psi(t))x=0.
\]
Since $\Psi(t)\to\infty$ as $t\to\infty$ and $\Psi([0,\infty))=[0,\infty)$,
this is equivalent to $\lim_{s\to\infty}\mathcal T(s)x=0$.
Hence $X_0^\delta=X_0$.

\smallskip
\noindent\emph{Step 2: invariance of periodic points.}
If $x\in X_p^\delta$, then $\mathcal S_\delta(t)x=x$ for some $t>0$, and thus
$\mathcal T(\Psi(t))x=x$ with $\Psi(t)>0$, so $x\in X_p$.
Conversely, if $x\in X_p$, there exists $s>0$ with $\mathcal T(s)x=x$; choosing
$t=\Psi^{-1}(s)$ yields
\[
\mathcal S_\delta(t)x=\mathcal T(\Psi(t))x=\mathcal T(s)x=x,
\]
so $x\in X_p^\delta$.
Therefore $X_p^\delta=X_p$.

\smallskip
\noindent\emph{Step 3: invariance of $X_\infty^\delta$.}
Fix $x\in X$ and $\varepsilon>0$.
Assume $x\in X_\infty^\delta$. Then there exist $y\in X$ and $t>0$ such that
$\|y\|<\varepsilon$ and $\|\mathcal S_\delta(t)y-x\|<\varepsilon$.
Since $\mathcal S_\delta(t)=\mathcal T(\Psi(t))$, we get
$\|\mathcal T(\Psi(t))y-x\|<\varepsilon$, hence $x\in X_\infty$.
Conversely, if $x\in X_\infty$, there exist $y\in X$ and $s>0$ such that
$\|y\|<\varepsilon$ and $\|\mathcal T(s)y-x\|<\varepsilon$.
Setting $t=\Psi^{-1}(s)$ yields
\[
\|\mathcal S_\delta(t)y-x\|
=
\|\mathcal T(\Psi(t))y-x\|
=
\|\mathcal T(s)y-x\|
<\varepsilon,
\]
so $x\in X_\infty^\delta$.
Thus $X_\infty^\delta=X_\infty$.

\smallskip
\noindent\emph{Step 4: hypercyclicity and chaos.}
By Remark~\ref{rem:orbit_identity}, orbit sets coincide for $\mathcal S_\delta$ and
$\mathcal T$, hence the existence of a dense orbit is equivalent for the two families.
Together with $X_p^\delta=X_p$, this yields the equivalence of chaos.
\end{proof}

\subsection{Infinitesimal generators and time reparametrization}

Let $\mathcal A_\delta$ denote the $\delta$--generator of $\mathcal S_\delta$ and let
$\mathcal B$ denote the classical generator of $\mathcal T$.

\begin{lemma}[Identification of generators]\label{lem:generator_identification}
One has
\[
D(\mathcal A_\delta)=D(\mathcal B),
\qquad
\mathcal A_\delta x=\mathcal Bx
\quad\text{for all }x\in D(\mathcal A_\delta).
\]
\end{lemma}

\begin{proof}
Fix $x\in X$ and relate $s$ and $t$ by $s=\Psi(t)=t^\delta/\delta$, i.e.\
$t=\Psi^{-1}(s)=(\delta s)^{1/\delta}$.
Then
\[
\frac{\mathcal T(s)x-x}{s}
=
\frac{\mathcal S_\delta(t)x-x}{t^\delta/\delta}.
\]
Letting $s\to0^+$ is equivalent to letting $t\to0^+$, and the difference quotients
coincide for corresponding values.
Hence the limit defining $\mathcal Bx$ exists if and only if the limit defining
$\mathcal A_\delta x$ exists, and the two values are equal when they exist.
\end{proof}

\begin{remark}\label{rem:generator_dynamics}
Lemma~\ref{lem:generator_identification} is the infinitesimal counterpart of the orbit
identity.
In particular, spectral conditions used in linear dynamics (for instance those appearing
in chaos criteria) may be verified for $\mathcal A_\delta$ exactly as in the classical
generator framework.
\end{remark}

\subsection{A Desch--Schappacher--Webb criterion in conformable time}

We now formulate a chaoticity criterion for $C_0$--$\delta$--semigroups obtained by
transporting the classical theorem of Desch--Schappacher--Webb \cite{desch1997} through
the conformable clock.

\begin{theorem}[Desch--Schappacher--Webb criterion for $C_0$--$\delta$--semigroups]
\label{thm:DSW_delta}
Let $X$ be separable and let $\mathcal S_\delta$ be a $C_0$--$\delta$--semigroup on $X$
with $\delta$--generator $\mathcal A_\delta$.
Assume that there exist:
\begin{itemize}
\item an open set $V\subset\sigma_p(\mathcal A_\delta)$ such that $V\cap i\R\neq\emptyset$;
\item for each $\lambda\in V$, an eigenvector $x_\lambda\in X\setminus\{0\}$ satisfying
$\mathcal A_\delta x_\lambda=\lambda x_\lambda$;
\item for every $\phi\in X^\ast$, the scalar map
\(
\lambda\mapsto \langle \phi, x_\lambda\rangle
\)
is analytic on $V$;
\item and if $\langle \phi, x_\lambda\rangle\equiv 0$ on $V$, then $\phi=0$.
\end{itemize}
Then $\mathcal S_\delta$ is $\delta$--chaotic (and hence $\delta$--hypercyclic).
\end{theorem}

\begin{proof}
Let $\mathcal T$ be defined by \eqref{eq:defT}.
By Lemma~\ref{lem:T_C0}, $\mathcal T$ is a $C_0$--semigroup, and by
Lemma~\ref{lem:generator_identification} its generator coincides with $\mathcal A_\delta$
(with the same domain).
Therefore the assumptions are exactly those required by the classical
Desch--Schappacher--Webb criterion for $\mathcal T$ \cite{desch1997}.
It follows that $\mathcal T$ is chaotic in the classical sense.

Finally, Lemma~\ref{lem:invariance_sets} shows that both the existence of a dense orbit
and the density of periodic points are invariant under the clock change
$s=\Psi(t)$.
Consequently, $\mathcal S_\delta$ is $\delta$--chaotic.
\end{proof}

\begin{remark}[Dynamical meaning of the clock change]\label{rem:clock_dynamics}
The identity $\mathcal S_\delta(t)=\mathcal T(\Psi(t))$ implies that orbit-based
properties depend only on the orbit sets, hence they remain unchanged under conformable
reparametrization.
As a result, chaoticity of $\mathcal S_\delta$ can be checked using classical criteria
for $\mathcal T$ (equivalently for $\mathcal A_\delta$), and conversely.
\end{remark}

\section{Applications}\label{sec7}

We illustrate the abstract results on two model problems.
Both examples rely on the same structural mechanism: a conformable evolution posed
on a weighted space can be reduced to a classical evolution by a nonlinear change
of the \emph{space} variable.
In our setting, this reduction is implemented by the unitary map $U$ (introduced
below) which identifies the weighted space with a standard $L^2$ space.
As a consequence, the conformable solution operator is not only a conformable
semigroup in the sense of Section~\ref{sec2}, but is also \emph{conjugate} to a
classical $C_0$--semigroup.
For orbit-based dynamics this is decisive: dense orbits and periodic points are
preserved under such conjugacies, so chaoticity of the conformable semigroup is
inherited directly from the classical one, for instance via the
Desch--Schappacher--Webb theory \cite{desch1997}.

\subsection{A conformable diffusion--transport flow on $X^\delta=L^{2,\delta}(0,1)$}\label{subsec:app1}

Fix $\delta\in(0,1]$.
We denote by $\partial_x^\delta$ the conformable derivative with respect to the
space variable $x>0$, namely
\[
\partial_x^\delta u(x):=\lim_{\varepsilon\to 0}\frac{u\bigl(x+\varepsilon x^{1-\delta}\bigr)-u(x)}{\varepsilon},
\]
whenever the limit exists.
For $u\in C^1((0,1))$, this derivative is local and satisfies the pointwise
representation (see \cite{khalil2014,abdeljawad2015,abdeljawad2016})
\begin{equation}\label{eq:conf_repr_final}
\partial_x^\delta u(x)=x^{1-\delta}u'(x),\qquad x\in(0,1).
\end{equation}

We work on the conformable Lebesgue space
\[
X^\delta:=L^{2,\delta}(0,1)
=\Bigl\{f:(0,1)\to\mathbb C:\ \int_0^1 |f(x)|^2\,x^{\delta-1}\,dx<\infty\Bigr\},
\]
endowed with the norm
\(
\|f\|_{X^\delta}^2=\int_0^1 |f(x)|^2 x^{\delta-1}\,dx.
\)
This is the natural Hilbert framework associated with the conformable weight
$x^{\delta-1}dx$.

\begin{lemma}[Spatial clock change and unitary transfer]\label{lem:spatial_unitary}
Define the change of variables
\[
\xi=x^\delta\in(0,1),\qquad x=\xi^{1/\delta},
\]
and the linear map $U:X^\delta\to L^2(0,1)$ by
\[
(Uf)(\xi):=\delta^{-1/2}\,f(\xi^{1/\delta}),\qquad \xi\in(0,1).
\]
Then $U$ is unitary. In particular, for every $f\in X^\delta$,
\begin{equation}\label{eq:measure_change_final}
\int_0^1 |f(x)|^2\,x^{\delta-1}\,dx
=
\int_0^1 |(Uf)(\xi)|^2\,d\xi.
\end{equation}
Moreover, if $u(\cdot,t)$ is sufficiently smooth and we set
\[
w(\xi,t):=u(\xi^{1/\delta},t),
\]
then for each $t\ge 0$ and every $\xi\in(0,1)$,
\begin{equation}\label{eq:conf_to_classical_final}
\partial_x^\delta u(\xi^{1/\delta},t)=\delta\,\partial_\xi w(\xi,t),
\qquad
\partial_x^\delta\!\bigl(\partial_x^\delta u\bigr)(\xi^{1/\delta},t)=\delta^2\,\partial_{\xi\xi} w(\xi,t).
\end{equation}
\end{lemma}

\begin{proof}
\emph{Unitarity.}
With $\xi=x^\delta$ we have $d\xi=\delta x^{\delta-1}\,dx$, hence
$x^{\delta-1}\,dx=\delta^{-1}d\xi$.
Therefore, for $f\in X^\delta$,
\[
\|f\|_{X^\delta}^2
=
\int_0^1 |f(x)|^2\,x^{\delta-1}\,dx
=
\frac1\delta\int_0^1 \bigl|f(\xi^{1/\delta})\bigr|^2\,d\xi
=
\int_0^1 \bigl|\delta^{-1/2}f(\xi^{1/\delta})\bigr|^2\,d\xi
=
\|Uf\|_{L^2(0,1)}^2,
\]
which proves \eqref{eq:measure_change_final} and shows that $U$ is an isometry.
The inverse map is explicit:
\[
(U^{-1}g)(x)=\delta^{1/2}\,g(x^\delta),\qquad x\in(0,1),
\]
hence $U$ is surjective and therefore unitary.

\medskip
\emph{Derivative identities.}
Assume $u(\cdot,t)\in C^1((0,1))$ and define $w(\xi,t)=u(\xi^{1/\delta},t)$.
Since $x=\xi^{1/\delta}$,
\[
\partial_\xi x=\frac{1}{\delta}\xi^{\frac1\delta-1}
=\frac{1}{\delta}x^{1-\delta}.
\]
By the chain rule,
\[
\partial_\xi w(\xi,t)
=
u_x(x,t)\,\partial_\xi x
=
u_x(\xi^{1/\delta},t)\,\frac{1}{\delta}\,x^{1-\delta}.
\]
Using \eqref{eq:conf_repr_final} we infer
\[
\partial_x^\delta u(\xi^{1/\delta},t)
=
x^{1-\delta}u_x(\xi^{1/\delta},t)
=
\delta\,\partial_\xi w(\xi,t),
\]
which is the first identity in \eqref{eq:conf_to_classical_final}.
Applying the same argument once more yields the second identity.
\end{proof}

\begin{proposition}[Unitary equivalence of conformable and classical generators]\label{prop:conjugacy_operator}
Let $a,b,c>0$.
Consider the operator $A_\delta$ on $X^\delta$ defined by
\[
(A_\delta f)(x):=a\,\partial_x^\delta\!\bigl(\partial_x^\delta f\bigr)(x)+b\,\partial_x^\delta f(x)+c\,f(x),
\qquad x\in(0,1),
\]
with domain
\[
D(A_\delta):=\Bigl\{f\in X^\delta:\ f,\ \partial_x^\delta f,\ \partial_x^\delta(\partial_x^\delta f)\in X^\delta,
\ \text{and}\ \lim_{x\downarrow 0}f(x)=0\Bigr\}.
\]
Define the classical operator $\widetilde A$ on $L^2(0,1)$ by
\[
(\widetilde A g)(\xi):=\widetilde a\,g''(\xi)+\widetilde b\,g'(\xi)+c\,g(\xi),
\qquad
\widetilde a:=a\delta^2,\quad \widetilde b:=b\delta,
\]
with domain
\[
D(\widetilde A):=\{g\in W^{2,2}(0,1):\ g(0)=0\}.
\]
Then, on the natural core where pointwise computations are justified, one has
\begin{equation}\label{eq:conjugacy_final}
\widetilde A = U A_\delta U^{-1}.
\end{equation}
In particular, if $\widetilde A$ generates a classical $C_0$--semigroup $\{S(t)\}_{t\ge 0}$ on $L^2(0,1)$,
then $A_\delta$ generates a conformable semigroup $\mathcal S_\delta=\{\mathcal S_\delta(t)\}_{t\ge0}$
on $X^\delta$ given by
\[
\mathcal S_\delta(t)=U^{-1}S(t)U,\qquad t\ge 0.
\]
\end{proposition}

\begin{proof}
Let $g\in D(\widetilde A)$ and set $f=U^{-1}g$, i.e.\ $f(x)=\delta^{1/2}g(x^\delta)$.
Write $\xi=x^\delta$, so that $f(x)=\delta^{1/2}g(\xi)$.
By Lemma~\ref{lem:spatial_unitary},
\[
\partial_x^\delta f(x)=\delta^{3/2}g'(\xi),
\qquad
\partial_x^\delta\!\bigl(\partial_x^\delta f\bigr)(x)=\delta^{5/2}g''(\xi).
\]
Therefore,
\[
(A_\delta f)(x)
=
\delta^{1/2}\Bigl((a\delta^2)g''(\xi)+(b\delta)g'(\xi)+c\,g(\xi)\Bigr).
\]
Applying $U$ and using $(Uf)(\xi)=\delta^{-1/2}f(\xi^{1/\delta})$ gives
\[
(UA_\delta U^{-1}g)(\xi)
=
(a\delta^2)g''(\xi)+(b\delta)g'(\xi)+c\,g(\xi)
=
(\widetilde A g)(\xi),
\]
which proves \eqref{eq:conjugacy_final}.
The semigroup relation follows from the standard correspondence under unitary equivalence
(see, e.g., \cite{engel2000,pazy1983}).
\end{proof}

\begin{theorem}[Chaoticity of the conformable diffusion--transport semigroup]\label{thm:conformable_example412}
Consider the conformable drift--diffusion equation on $(0,1)$
\begin{equation}\label{eq:conf_pde}
\begin{cases}
\partial_t u(x,t)
=
a\,\partial_x^\delta\!\bigl(\partial_x^\delta u\bigr)(x,t)
+
b\,\partial_x^\delta u(x,t)
+
c\,u(x,t),
& x\in(0,1),\ t\ge 0,\\[1.5mm]
u(0,t)=0, & t\ge 0,\\
u(x,0)=f(x), & x\in(0,1),
\end{cases}
\end{equation}
where $f\in X^\delta=L^{2,\delta}(0,1)$ and the parameters $a,b,c>0$ satisfy
\[
c<\frac{b^2}{2a}<1.
\]
Let $A_\delta$ be as in Proposition~\ref{prop:conjugacy_operator} and let
$\mathcal S_\delta=\{\mathcal S_\delta(t)\}_{t\ge 0}$ be the conformable semigroup on $X^\delta$
generated by $A_\delta$. Then $\mathcal S_\delta$ is chaotic on $X^\delta$.
In particular, there exists $f\in X^\delta$ such that the orbit
$\{\mathcal S_\delta(t)f:\ t\ge 0\}$ is dense in $X^\delta$, and the set of periodic points of
$\mathcal S_\delta$ is dense in $X^\delta$.

Moreover, under the spatial change of variables $\xi=x^\delta$ and $w(\xi,t)=u(\xi^{1/\delta},t)$,
the mild solution of \eqref{eq:conf_pde} satisfies
\[
u(\cdot,t)=\mathcal S_\delta(t)f
\quad\Longleftrightarrow\quad
w(\cdot,t)=S(t)\,(Uf),
\]
where $\{S(t)\}$ is the classical $C_0$--semigroup generated on $L^2(0,1)$ by $\widetilde A$.
\end{theorem}

\begin{proof}
\emph{Step 1: Reduction to the classical evolution.}
Let $u$ be sufficiently regular and define $w(\xi,t)=u(\xi^{1/\delta},t)$ with $\xi=x^\delta$.
Lemma~\ref{lem:spatial_unitary} yields \eqref{eq:conf_to_classical_final}, hence
\eqref{eq:conf_pde} is equivalent to
\[
\partial_t w(\xi,t)
=
(a\delta^2)\,\partial_{\xi\xi}w(\xi,t)
+
(b\delta)\,\partial_\xi w(\xi,t)
+
c\,w(\xi,t),
\qquad \xi\in(0,1),\ t\ge 0,
\]
with $w(0,t)=0$ and $w(\cdot,0)=Uf$.

\medskip
\emph{Step 2: Chaoticity in the classical picture.}
Let $\widetilde A$ be as in Proposition~\ref{prop:conjugacy_operator}.
Under the drift--diffusion condition
\[
c<\frac{\widetilde b^2}{2\widetilde a}<1,
\]
the classical semigroup generated by $\widetilde A$ is chaotic; this follows from
Desch--Schappacher--Webb type criteria \cite{desch1997}.
Since
\[
\frac{\widetilde b^2}{2\widetilde a}
=
\frac{(b\delta)^2}{2(a\delta^2)}
=
\frac{b^2}{2a},
\]
our assumption $c<\frac{b^2}{2a}<1$ is exactly the required one, so $\{S(t)\}$ is chaotic.

\medskip
\emph{Step 3: Transfer to the conformable semigroup.}
By Proposition~\ref{prop:conjugacy_operator},
\[
\mathcal S_\delta(t)=U^{-1}S(t)U,\qquad t\ge0.
\]
Since $U$ is unitary, it preserves density of sets and maps periodic points to periodic points.
Therefore chaoticity of $\{S(t)\}$ implies chaoticity of $\mathcal S_\delta$ on $X^\delta$.

\medskip
Finally, the mild-solution correspondence is precisely the intertwining identity
$U(\mathcal S_\delta(t)f)=S(t)(Uf)$, which is equivalent to
$w(\cdot,t)=S(t)(Uf)$ for $w(\xi,t)=u(\xi^{1/\delta},t)$.
\end{proof}

\begin{remark}\label{rem:what_conformable_here}
In this example the conformable drift--diffusion operator becomes classical after the spatial clock
$\xi=x^\delta$.
The rescaling $(a,b)\mapsto(a\delta^2,b\delta)$ leaves the ratio $\frac{b^2}{2a}$ unchanged, which
explains why the same inequality controls chaoticity in both formulations.
\end{remark}

\subsection{A conjugacy principle: transporting chaos from $L^2(0,1)$ to $X^\delta$}\label{subsec:app2}

Theorem~\ref{thm:conformable_example412} rests on two structural facts:
(i) the spatial clock $\xi=x^\delta$ yields a unitary identification between $X^\delta$
and $L^2(0,1)$ (Lemma~\ref{lem:spatial_unitary}), and
(ii) the conformable generator $A_\delta$ is unitarily equivalent to the classical
drift--diffusion operator $\widetilde A$ (Proposition~\ref{prop:conjugacy_operator}).
Taken together, these statements imply that the two semigroups are intertwined by a
homeomorphism; consequently, orbit-based properties are invariant under this conjugacy.
We isolate this consequence in a form that can be reused in other conformable models.

\begin{theorem}[Chaoticity transfer under unitary spatial equivalence]\label{thm:app_diffusion_transport_2}
Fix $\delta\in(0,1]$ and assume the setting of Proposition~\ref{prop:conjugacy_operator}.
Let $U:X^\delta\to L^2(0,1)$ be the unitary operator defined in Lemma~\ref{lem:spatial_unitary},
and assume that
\begin{equation}\label{eq:app2_generator_equivalence}
\widetilde A = U A_\delta U^{-1}.
\end{equation}
Denote by $\{S(t)\}_{t\ge0}$ the $C_0$--semigroup generated by $\widetilde A$ on $L^2(0,1)$, and by
$\mathcal S_\delta=\{\mathcal S_\delta(t)\}_{t\ge0}$ the conformable semigroup generated by $A_\delta$
on $X^\delta$. Then the following hold.
\begin{enumerate}
\item \emph{(Semigroup conjugacy)} For every $t\ge0$,
\begin{equation}\label{eq:app2_Sdelta_conjugacy}
\mathcal S_\delta(t)=U^{-1}S(t)U,
\qquad\text{equivalently}\qquad
U\mathcal S_\delta(t)=S(t)U .
\end{equation}

\item \emph{(Orbit correspondence)} For each $f\in X^\delta$,
\[
U\bigl(\{\mathcal S_\delta(t)f:\ t\ge0\}\bigr)
=
\{S(t)(Uf):\ t\ge0\},
\]
and therefore
\[
\overline{\{\mathcal S_\delta(t)f:\ t\ge0\}}=X^\delta
\quad\Longleftrightarrow\quad
\overline{\{S(t)(Uf):\ t\ge0\}}=L^2(0,1).
\]

\item \emph{(Periodic points)} Let
\begin{align*}
\mathrm{Per}(\mathcal S_\delta)
&=
\{f\in X^\delta:\ \exists\,t_0>0\ \text{such that }\mathcal S_\delta(t_0)f=f\},\\
\mathrm{Per}(S)
&=
\{g\in L^2(0,1):\ \exists\,t_0>0\ \text{such that }S(t_0)g=g\}.
\end{align*}
Then $U(\mathrm{Per}(\mathcal S_\delta))=\mathrm{Per}(S)$ and hence
\[
\overline{\mathrm{Per}(\mathcal S_\delta)}=X^\delta
\quad\Longleftrightarrow\quad
\overline{\mathrm{Per}(S)}=L^2(0,1).
\]

\item \emph{(Chaoticity transfer)} If $\{S(t)\}$ is chaotic on $L^2(0,1)$
(for instance under the drift--diffusion regime used in
Theorem~\ref{thm:conformable_example412}, as ensured via criteria of
Desch--Schappacher--Webb type \cite{desch1997}),
then $\mathcal S_\delta$ is chaotic on $X^\delta$.
\end{enumerate}
\end{theorem}

\begin{proof}
Define $\widehat S(t):=U^{-1}S(t)U$ on $X^\delta$.
Then $\{\widehat S(t)\}$ is a $C_0$--semigroup on $X^\delta$ and, for $f\in D(A_\delta)$,
\[
\lim_{t\downarrow0}\frac{\widehat S(t)f-f}{t}
=
U^{-1}\widetilde A(Uf)
=
A_\delta f,
\]
using \eqref{eq:app2_generator_equivalence}.
Hence $\{\widehat S(t)\}$ is generated by $A_\delta$.
By uniqueness of $C_0$--semigroups with a given generator, we obtain
$\widehat S(t)=\mathcal S_\delta(t)$, proving \eqref{eq:app2_Sdelta_conjugacy}.

The orbit and periodic-point correspondences follow by applying $U$ to the identities
$\mathcal S_\delta(t)=U^{-1}S(t)U$ and $U\mathcal S_\delta(t)=S(t)U$.
Since $U$ is unitary, it preserves closures and denseness, yielding items (2)--(3).
Finally, chaoticity is the conjunction of hypercyclicity and density of periodic points,
hence item (4) follows immediately.
\end{proof}

\begin{remark}
Theorem~\ref{thm:app_diffusion_transport_2} provides a reusable template:
whenever a conformable generator can be reduced to a classical one through a unitary
(or more generally topological) conjugacy, orbit-type dynamical properties are
automatically transported to the conformable semigroup.
\end{remark}

\subsection{A conformable transport model on weighted spaces: reduction to translations and dynamical consequences}
\label{subsec:app_conformable_transport}

In this subsection we show that the conformable evolution problem
\begin{equation}\label{eq:app_conf_transport_pde}
\begin{cases}
\partial_t u(t,x)=\partial_x^\alpha u(t,x), & t,x\in\R_+,\\
u(0,x)=f(x),
\end{cases}
\end{equation}
can be reduced explicitly to the classical translation semigroup by a nonlinear
change of the space variable.
Consequently, orbit-based dynamical properties (hypercyclicity and Devaney chaos)
for the solution semigroup are inherited from the well-understood dynamics of
translations on suitable weighted spaces.

\medskip
\noindent\textbf{Functional setting.}
Let $\alpha\in(0,1]$ and let $\rho_\alpha:\R_+\to(0,\infty)$ be a weight.
We work on the Banach space
\[
\widetilde X \in \Bigl\{\, C_{0,\rho_\alpha}(\R_+;\C),\; L^p_{\rho_\alpha}(\R_+;\C)\ (1\le p<\infty)\Bigr\},
\]
equipped with the standard weighted norms.
Throughout this subsection, $\partial_x^\alpha$ denotes the conformable derivative with respect
to $x$ in the sense recalled in Section~\ref{sec2}. In particular, for $v\in C^1((0,\infty))$,
\begin{equation}\label{eq:app_conf_repr_x}
\partial_x^\alpha v(x)=x^{1-\alpha}v'(x),\qquad x>0.
\end{equation}

\subsubsection*{A nonlinear space clock}

Define the \emph{conformable spatial clock}
\begin{equation}\label{eq:app_space_clock}
\psi:\R_+\to\R_+,\qquad \psi(x):=\frac{x^\alpha}{\alpha}.
\end{equation}
Then $\psi$ is a $C^1$--diffeomorphism of $\R_+$ onto itself with inverse
\[
\psi^{-1}(\xi)=(\alpha\xi)^{1/\alpha},\qquad \xi\ge0.
\]
Given a sufficiently regular solution $u(t,x)$ of \eqref{eq:app_conf_transport_pde}, define
\begin{equation}\label{eq:app_w_def}
w(t,\xi):=u\bigl(t,\psi^{-1}(\xi)\bigr),\qquad t,\xi\ge0.
\end{equation}
A direct chain-rule computation using \eqref{eq:app_conf_repr_x} yields
\[
\partial_\xi w(t,\xi)=\partial_x^\alpha u\bigl(t,\psi^{-1}(\xi)\bigr),
\]
so the equation $\partial_t u=\partial_x^\alpha u$ is equivalent to the classical
transport equation
\begin{equation}\label{eq:app_translation_pde}
\partial_t w(t,\xi)=\partial_\xi w(t,\xi),\qquad t,\xi\ge0,
\end{equation}
with $w(0,\xi)=(Qf)(\xi)$, where $Q$ is the composition operator introduced below.
Equation \eqref{eq:app_translation_pde} is solved by translations:
\begin{equation}\label{eq:app_translation_solution}
w(t,\xi)=w(0,\xi+t),\qquad t,\xi\ge0.
\end{equation}

\subsubsection*{Conjugacy with the translation semigroup}

Let $\widetilde\rho_\alpha$ be the transported weight defined by
\begin{equation}\label{eq:app_weight_transport}
\widetilde\rho_\alpha(\xi):=\rho_\alpha\bigl(\psi^{-1}(\xi)\bigr),\qquad \xi\ge0.
\end{equation}
Consider the linear map
\begin{equation}\label{eq:app_Q_def}
Q:\widetilde X\to \widetilde Y,\qquad (Qf)(\xi):=f\bigl(\psi^{-1}(\xi)\bigr),
\end{equation}
where $\widetilde Y$ denotes the corresponding weighted space in the $\xi$--variable:
\[
\widetilde Y \in \Bigl\{\, C_{0,\widetilde\rho_\alpha}(\R_+;\C),\;
L^p_{\widetilde\rho_\alpha}(\R_+;\C)\ (1\le p<\infty)\Bigr\}.
\]
Let $\{W(t)\}_{t\ge0}$ be the classical translation semigroup on $\widetilde Y$,
\begin{equation}\label{eq:app_W_def}
(W(t)g)(\xi):=g(\xi+t),\qquad \xi,t\ge0,
\end{equation}
and define $\{S_\alpha(t)\}_{t\ge0}$ on $\widetilde X$ by
\begin{equation}\label{eq:app_Salpha_def}
(S_\alpha(t)f)(x):=f\bigl(\psi^{-1}(\psi(x)+t)\bigr),\qquad x,t\ge0.
\end{equation}
Then $u(t,\cdot)=S_\alpha(t)f$ is the (mild) solution of \eqref{eq:app_conf_transport_pde}.

\begin{proposition}[Explicit solution semigroup and conjugacy]\label{prop:app_conjugacy_translation}
Let $\alpha\in(0,1]$ and let $\widetilde X$ be as above.
Then $\{S_\alpha(t)\}_{t\ge0}$ defines a strongly continuous semigroup on $\widetilde X$ and satisfies
\begin{equation}\label{eq:app_conjugacy_identity}
Q\,S_\alpha(t)=W(t)\,Q,\qquad t\ge0,
\end{equation}
equivalently,
\begin{equation}\label{eq:app_conjugacy_equiv}
S_\alpha(t)=Q^{-1}W(t)Q,\qquad t\ge0.
\end{equation}
In particular, for each $f\in\widetilde X$,
\[
Q\bigl(\{S_\alpha(t)f:\ t\ge0\}\bigr)=\{W(t)(Qf):\ t\ge0\}.
\]
\end{proposition}

\begin{proof} Let $f\in\widetilde X$ and $\xi\ge0$. Using \eqref{eq:app_Salpha_def} and \eqref{eq:app_Q_def},
\begin{align*}
(QS_\alpha(t)f)(\xi)&=(S_\alpha(t)f)\bigl(\psi^{-1}(\xi)\bigr)\\
&=f\Bigl(\psi^{-1}\bigl(\psi(\psi^{-1}(\xi))+t\bigr)\Bigr)\\
&=f\bigl(\psi^{-1}(\xi+t)\bigr)\\
&=(Qf)(\xi+t)\\
&=(W(t)Qf)(\xi),
\end{align*}

which gives \eqref{eq:app_conjugacy_identity} and hence \eqref{eq:app_conjugacy_equiv}. Strong continuity follows from that of $\{W(t)\}$ and the boundedness of $Q,Q^{-1}$.
\end{proof}

\begin{corollary}[Invariance of orbit-based dynamics]\label{cor:app_dynamics_invariance}
With the above notation, the following are equivalent:
\begin{enumerate}
\item $\{S_\alpha(t)\}_{t\ge0}$ is hypercyclic (resp.\ chaotic) on $\widetilde X$;
\item $\{W(t)\}_{t\ge0}$ is hypercyclic (resp.\ chaotic) on $\widetilde Y$.
\end{enumerate}
Moreover, $f\in\widetilde X$ is hypercyclic for $\{S_\alpha(t)\}$ if and only if
$Qf$ is hypercyclic for $\{W(t)\}$, and the same correspondence holds for periodic points.
\end{corollary}

\begin{proof}
The map $Q$ is a homeomorphism and \eqref{eq:app_conjugacy_identity} sends orbits of $\{S_\alpha(t)\}$
onto orbits of $\{W(t)\}$.
Therefore density of orbits and density of periodic points are preserved under $Q$.
\end{proof}

\begin{remark}
Corollary~\ref{cor:app_dynamics_invariance} reduces orbit-based dynamics of
\eqref{eq:app_conf_transport_pde} to the translation semigroup on the transported space.
Consequently, criteria for translations on weighted spaces (expressed in terms of the
asymptotic behavior of $\widetilde\rho_\alpha$) can be transported back to the original weight
$\rho_\alpha$ via \eqref{eq:app_weight_transport}.
\end{remark}

\section*{Conclusion}

This work shows that conformable time evolution can be understood entirely within the classical $C_0$--semigroup framework. The key point is that, for a fixed order $\delta\in(0,1]$, the conformable derivative is local and reduces on smooth functions to a weighted classical derivative. This locality leads to a rigid structural mechanism: every $C_0$--$\delta$--semigroup $\mathcal S_\delta$ is exactly a classical $C_0$--semigroup $\mathcal T$ observed through the nonlinear time change $\Psi(t)=t^\delta/\delta$, namely
\[
\mathcal S_\delta(t)=\mathcal T(\Psi(t)).
\]
The correspondence is not merely formal. We proved that the $\delta$--generator of $\mathcal S_\delta$ coincides with the generator of $\mathcal T$ on the same domain, and that conformable mild solutions are in one-to-one correspondence with classical mild solutions under $s=\Psi(t)$. In addition, the functional setting induced by the conformable integral, based on the measure $t^{\delta-1}\,dt$, does not introduce a genuinely new scale: conformable Lebesgue and Sobolev spaces are explicitly (and in the relevant formulations, isometrically) equivalent to their classical counterparts via the same change of variables. This identification allows one to transport estimates and well-posedness results directly, including generation statements in Hilbert spaces through the Lumer--Phillips theorem.

A central dynamical consequence is that conformable reparametrization does not create new trajectories: it only changes the speed at which orbits are traversed. Hence orbit-based properties are invariant. We established that $\delta$--hypercyclicity and $\delta$--chaos for $\mathcal S_\delta$ coincide with the classical notions for $\mathcal T$, and we derived a conformable Desch--Schappacher--Webb criterion by transporting the classical theorem through the clock relation. The applications developed in the final part of the paper further emphasize this transfer principle: after a suitable nonlinear change of the spatial variable, the conformable generators become unitarily equivalent to classical drift--diffusion or translation generators, so chaoticity follows immediately from known results without new spectral computations.

Overall, the results clarify the conceptual status of conformable models in semigroup theory. Unlike nonlocal fractional dynamics, where memory effects are intrinsic, conformable evolution remains local and its ``fractional'' aspect is fully encoded by deterministic changes of variables. From a practical viewpoint, this provides a systematic methodology: whenever a conformable problem can be reduced to a classical one by an explicit time or space clock, qualitative properties and dynamical criteria can be imported directly. Natural continuations include extending the conjugacy principle to broader classes of weights and boundary conditions, studying robustness under perturbations, and identifying settings where additional modeling features (e.g.\ genuinely nonlocal terms) are required to produce dynamics that go beyond a pure reparametrization effect.

\end{document}